\documentclass[reqno,dvipsnames]{amsart}

\usepackage{amssymb,amsmath,tikz,amsthm,hyperref}
\usepackage{longtable}
\usepackage{tablefootnote}
\usepackage{enumerate}
\usepackage{listings}
\usepackage[margin = 1in]{geometry}
\usepackage{xcolor}

\lstdefinestyle{mystyle}{
	commentstyle=\color{teal},
    keywordstyle=\color{Green},
    stringstyle=\color{Maroon},
	basicstyle=\ttfamily\footnotesize,
    breakatwhitespace=false,         
    breaklines=true,                 
    captionpos=b,                    
    keepspaces=true,                 
    showspaces=false,                
    showstringspaces=false,
    showtabs=false,                  
    tabsize=2
}

\renewcommand{\geq}{\geqslant}
\renewcommand{\leq}{\leqslant}

\newcommand{\R}{\mathbb{R}}
\newcommand{\Z}{\mathbb{Z}}

\newcommand{\C}{\mathbb{C}}

\newcommand{\inv}[1]{{#1}^{-1}}

\newtheorem{theorem}{Theorem}[section]
\newtheorem{proposition}{Proposition}[section]
\newtheorem{lemma}{Lemma}[section]

\newtheorem*{theorem*}{Theorem}
\newtheorem*{lemma*}{Lemma}
\newtheorem*{corollary*}{Corollary}
\newtheorem*{definition*}{Definition}
\newtheorem{definition}{Definition}

\title{The Number of Solutions to the Trinomial Thue Equation}
\author{Greg Knapp}

\begin{document}

\begin{abstract}
	In this paper, we study the number of integer pair solutions to the equation $|F(x,y)| = 1$ where $F(x,y) \in \Z[x,y]$ is an irreducible (over $\Z$) binary form with degree $n \geq 3$ and exactly three nonzero summands.  In particular, we improve Thomas' explicit upper bounds on the number of solutions to this equation (see \cite{Thomas2000}).  For instance, when $n \geq 219$, we show that there are no more than 32 integer pair solutions to this equation when $n$ is odd and no more than 40 integer pair solutions to this equation when $n$ is even, an improvement on Thomas' work in \cite{Thomas2000}, where he shows that there are no more than 38 such solutions when $n$ is odd and no more than 48 such solutions when $n$ is even.
\end{abstract}

\maketitle


\section{Introduction}

Axel Thue, in \cite{Thue1909}, showed that the equation
\begin{align}\label{eq:generalThueEq}
	F(x,y) = h
\end{align}
has only finitely many integer pair solutions when $F(x,y) \in \Z[x,y]$ is an irreducible (over $\Z$) homogeneous polynomial of degree $n \geq 3$ and $h \in \Z$ (for the purposes of this paper, only integer pair solutions will be considered).  A polynomial which is homogeneous in two variables is called a \emph{binary form} and the equation \eqref{eq:generalThueEq} is called \emph{Thue's equation}.  As a consequence, the equation
\begin{align}\label{eq:thueInequality}
	|F(x,y)| \leq h
\end{align}
also has finitely many solutions.  The equation \eqref{eq:thueInequality} is called \emph{Thue's Inequality}. Note that we are using all of the stated hypotheses here: if $F(x,y)$ has a linear factor for instance (say $rx - sy$ divides $F(x,y)$ for some relatively prime $r,s \in \Z$), then there are infinitely many integer pair solutions to $F(x,y) = 0$---namely, the integral multiples of the pair $(s,r)$---and hence, there are infinitely many integer pair solutions to \eqref{eq:thueInequality}.  If $n =\deg(F)$ has $n \leq 2$, then the family of Pell equations show that there can be infinitely many solutions to \eqref{eq:thueInequality} even when $F(x,y)$ is irreducible.  Thue's conclusion does hold if every irreducible factor of $F(x,y)$ has degree at least three, but we assume that $F(x,y)$ is irreducible for simplicity.  Several natural questions arise such as

\begin{enumerate}[\hspace{0.75cm}1.]
	\item How many solutions are there to \eqref{eq:thueInequality}?
	\item How large are solutions to \eqref{eq:thueInequality}?
	\item On which features of $F$ and $h$ do the solutions to \eqref{eq:thueInequality} depend?
\end{enumerate}

This paper largely handles the first question, though of course the second and third questions are related.  In particular, the number of nonzero summands of $F(x,y)$ significantly impacts the number of solutions to \eqref{eq:thueInequality}.\footnote{The rough reasoning for this is as follows: a solution $(p,q)$ to \eqref{eq:thueInequality} corresponds to a good rational approximation $p/q$ to a root of $f(X) := F(X,1)$.  The only roots of $f(X)$ which ought to allow good rational approximations are the real roots of $f(X)$ and it is the number of nonzero summands of $f(X)$ that controls the number of real roots of $f(X)$, as seen in Lemma 1 of \cite{Schmidt1987}}

\begin{definition}
	If a polynomial has exactly two nonzero summands, it is called a \emph{binomial}; if it has exactly three nonzero summands, it is called a \emph{trinomial}; and if it has exactly four nonzero summands, it is called a \emph{tetranomial}.
\end{definition}

Since the number of nonzero summands plays such an important role, authors such as Bennett \cite{Bennett2001}, Evertse \cite{Evertse1982}, Grundman and Wisniewski \cite{Grundman2013}, Hyyr\"o \cite{Hyyroe1964}, Mueller \cite{Mueller1987}, Mueller and Schmidt \cite{Mueller1986}, and Thomas \cite{Thomas2000} have examined binomial, trinomial, and tetranomial Thue equations in hopes to get a better handle on how the number of nonzero summands of $F(x,y)$ affect the number of solutions to Thue equations.\\

In this paper, we focus on improving explicit bounds for the number of solutions to the Thue equation
\begin{align}\label{eq:thueEq}
	|F(x,y)| = 1
\end{align}
in the particular case that $F(x,y)$ is a trinomial.  In this setting, Thomas, in \cite{Thomas2000}, showed that there are no more than $2v(n)w(n) + 8$ distinct integer pair solutions to $|F(x,y)| = 1$ when $F(x,y) \in \Z[x,y]$ is a trinomial irreducible binary form of degree $n \geq 3$ and $w(n)$ and $v(n)$ are piecewise defined as follows:

\begin{align*}
	v(n) = \begin{cases}
	3 & \text{if } n \text{ is odd}\\
	4 & \text{if } n \text{ is even}
	\end{cases}
\end{align*}
and
\begin{table}[h!]
\centering
	\begin{tabular}{| c || c | c | c | c | c | c | c | c | c |}\hline
	$n$    & 5\tablefootnote{There is an error in the proof of Lemma 4.1 in \cite{Thomas2000}: it is claimed that $\frac{b^t - 1}{b - 1} < b^t$ which is not the case for the choice of $b = 1.5$ when $n = 5$.  Tracing this error through to its conclusion, the author believes that this is not correctable.} & 6  & 7  & 8  & 9 & 10--11 & 12--16 & 17--37 & $\geq 38$\\\hline
	$w(n)$ & $27^\dag$ & 16 & 13 & 11 & 9 & 8      & 7      & 6      & 5\\\hline
	\end{tabular}
\end{table}

We are able to improve the bounds that Thomas provides and we have the following theorem.
\begin{theorem}\label{totalTrinomialSolutions}
	Let $F(x,y) = h_nx^n + h_kx^ky^{n - k} + h_0y^n$ where $h_n,h_k,h_0,n,k \in \Z$ with $0 < k < n$.  Suppose that $F(x,y)$ is irreducible over $\Z[x,y]$ and $n \geq 6$.  Then there are at most $2v(n)z(n) + 8$ distinct integer pair solutions to the equation $|F(x,y)| = 1$ where $v(n) = 3$ if $n$ is odd, $v(n) = 4$ if $n$ is even, and $z(n)$ is defined by the following table.
	\begin{center}
		\begin{tabular}{| c || c | c | c | c | c | c | c | c | c | c |}\hline
		$n$ 		& $6$  & $7$  & $8$  & $9$ & $10$--$11$ & $12$--$16$ & $17$--$38$ & $39$--$218$ & $\geq 219$\\\hline
		$z(n)$ & $15$ & $12$ & $11$ & $9$ & $8$     & $7$     & $6$     & $5$       & $4$\\\hline
		\end{tabular}
	\end{center}
\end{theorem}

This result is primarily derived from an improvement in efficiency to a counting technique associated to what is known as the gap principle (see Lemma \ref{GPImprovement}).


\section{Counting With Gaps}

The main technical accomplishment of this paper is the following version of a counting technique often used in conjunction with the ``gap principle.''

\begin{lemma} \label{GPImprovement}
	Suppose that $L, M, T, p,y_0,\dots,y_\ell \in \R_{>0}$ satisfy the following conditions:
	\begin{enumerate}
	 	\item $L \leq y_0 \leq \dots \leq y_\ell \leq M$
	 	\item $p > 2$
	 	\item $L^{p-2} > T$
	 	\item $y_{i+1} \geq \inv{T}y_i^{p-1}$ for each $0 \leq i < \ell$
	 \end{enumerate}
	 Then \[\ell \leq \frac{\log\left[\frac{\log(MT^{-1/(p-2)})}{\log\left(LT^{-1/(p-2)}\right)}\right]}{\log(p-1)}.\]
\end{lemma}

This lemma is comparable to Lemma 1 in \cite{Saradha2017}.  However, by fixing any $L> 0$, $p > 2$, $\ell \in \Z_{>0}$, and $0 < T < L^{p-2}$, then setting $y_0 = L$, $y_i = T^{-1}y_{i-1}^{p-1}$, and $M = y_\ell$, one can see that this upper bound is sharp where the upper bound in Lemma 1 of \cite{Saradha2017} is not.

\begin{proof}
	By induction, we have

	\begin{align*}
		M 	&\geq y_\ell \geq \frac{y_{\ell - 1}^{p-1}}{T} \geq \frac{\left(\frac{y_{\ell-2}^{p-1}}{T}\right)^{p-1}}{T} = \frac{y_{\ell - 2}^{(p-1)^2}}{T \cdot T^{p-1}} \geq \cdots\\
			&\cdots\geq \frac{y_0^{(p-1)^\ell}}{T^{\sum_{j=0}^{\ell-1} (p-1)^j}} = \frac{y_0^{(p-1)^\ell}}{T^{\frac{(p-1)^\ell - 1}{p-2}}} \geq \frac{L^{(p-1)^\ell}}{T^{\frac{(p-1)^\ell - 1}{p-2}}}
	\end{align*}
	and we multiply both sides of \[M \geq \frac{L^{(p-1)^\ell}}{T^{\frac{(p-1)^\ell - 1}{p-2}}}\] by $T^{-1/(p-2)}$ to get \[MT^{-1/(p-2)} \geq \left(LT^{-1/(p-2)}\right)^{(p-1)^\ell}.\]  Taking a log on both sides (and using the fact that $LT^{-1/(p-2)} > 1$) yields \[\frac{\log\left(MT^{-1/(p-2)}\right)}{\log\left(LT^{-1/(p-2)}\right)} \geq (p-1)^\ell\] and taking logs again and using the fact that $p > 2$ yields the desired inequality.
\end{proof}


\section{The Trinomial Thue Equation}

In this section, we follow Thomas in \cite{Thomas2000} for much of our reasoning.  However, we use different notation: the parameters which Thomas calls $u$ and $v$, we call $a$ and $b$ (this aligns with similar notation used in \cite{Akhtari2020} for instance, and also makes clear the difference between these parameters---whose choice will depend on $n$---and the values $u_n$ to be defined later and $v(n)$ defined in Theorem \ref{totalTrinomialSolutions}).  The parameters which Thomas calls $b$ and $b_0$, we will call $d$ and $d_0$ due to their relation to the degree of $F(x,y)$ and also to avoid conflict with the newly named $b$.

Throughout the remainder of this section, suppose that $F(x,y) = h_nx^n + h_kx^ky^{n-k} + h_0y^n$ where $h_n,h_k,h_0,n,k \in \Z$, $0 < k < n$, and $n \geq 6$.  Suppose further that $F(x,y)$ is irreducible over $\Z[x,y]$.  Let $H = \max(|h_n|, |h_k|, |h_0|)$ be the na\"ive height of $F(x,y)$.  Any time we refer to a ``solution,'' we specifically mean a solution to equation \eqref{eq:thueEq} in $\Z^2$.

We will not give a sophisticated bound on the number of solutions $(p,q)$ with $|pq| \leq 1$ and we will consider $(p,q)$ and $(-p,-q)$ to be equivalent solutions, spurring the following definition.

\begin{definition}
	A pair $(p,q) \in \Z^2$ is called \emph{regular} if $p \neq 0$, $q > 0$, and $|p| \neq q$.
\end{definition}

If there are $r$ regular solutions to \eqref{eq:thueEq}, then there will be at most $2r + 8$ distinct solutions since for every solution $(p,q)$ with $|pq| > 1$, either $(p,q)$ or $(-p,-q)$ is regular and there are at most 8 solutions with $|pq| \leq 1$.  From this fact and Theorem \ref{trinomialGeneralCount} below, Theorem \ref{totalTrinomialSolutions} will follow.

\begin{theorem}\label{trinomialGeneralCount}
	Equation \eqref{eq:thueEq} has at most $v(n)z(n)$ regular solutions where $v(n)$ and $z(n)$ are defined in Theorem \ref{totalTrinomialSolutions}.
\end{theorem}

More specifically, let $f(x) := F(x,1)$ and set $R_F$ to be the number of real roots of $f$.  We also wish to include certain critical points, so we make the following definition:

\begin{definition}
	A critical point $\tau \in \R$ of $g(x) \in \R[x]$ is \emph{proper} if there exists a neighborhood $U$ of $\tau$ for which $g''(x)g(x) > 0$ for all $x \in U \setminus\{\tau\}$.
\end{definition}

Now let $C_F$ to be the number of proper critical points of $f(x)$.  Setting $N_F$ to be the number of regular solutions to \eqref{eq:thueEq}, we will show the following theorem.

\begin{theorem}\label{trinomialExceptionalCount}
	Let $F(x,y)$ be a trinomial of degree $n \geq 6$.  Then \[N_F \leq z(n)R_F + \ell(n)C_F\] where $\ell(n)$ is defined by the following table.

	\begin{table}[h!]
		\centering
		\begin{tabular}{| c || c | c | c |}\hline
		$n$ & $6$--$7$ & $8$ & $\geq 9$\\\hline
		$\ell(n)$ & $4$ & $3$ & $2$\\\hline
		\end{tabular}
	\end{table}
\end{theorem}

We first show that Theorem \ref{trinomialExceptionalCount} implies Theorem \ref{trinomialGeneralCount}.  Since $\ell(n)$ is less than $z(n)$, we have that $z(n)R_F + \ell(n)C_F \leq z(n)(R_F + C_F)$ and one can check that that $R_F + C_F \leq v(n)$ with calculus, so we get Theorem \ref{trinomialGeneralCount}.

To prove Theorem \ref{trinomialExceptionalCount}, we need some additional setup.

\begin{definition}
	For a polynomial $g(x) \in \R[x]$, an \emph{exceptional point} of $f$ is either a real root or a proper critical point of $g(x)$
\end{definition}

Let $\mathcal{E}(f)$ be the set of exceptional points, $\tau_1 < \tau_2 < \dots < \tau_c$, of $f$.  Note that there exist improper critical points $\eta_1 < \eta_2 < \dots < \eta_{c-1}$ so that $\tau_1 < \eta_1 < \tau_2 < \eta_2 < \dots < \eta_{c-1} < \tau_c$.  Setting $\eta_0 = -\infty$ and $\eta_c = +\infty$, we can define $J_1 = (-\infty,\eta_1)$ and $J_i = [\eta_i,\eta_{i+1})$ for $1 \leq i \leq c$.

\begin{definition}
	A real number $\rho$ \emph{belongs to} $\tau_i$ (and $\tau_i$ \emph{belongs to} $\rho$) if $\rho \in J_i$.
\end{definition}

Observe that a regular pair $(p,q) \in \Z^2$ with $q \neq 0$ satisfying \eqref{eq:thueEq} corresponds to a rational number $\frac{p}{q}$ satisfying \[\left|f\left(\frac{p}{q}\right)\right| = \frac{1}{q^n}.\]  Moreover, because $|F(p,q)| = 1$ and $F(x,y) \in \Z[x,y]$, $p$ and $q$ must be relatively prime, so this correspondence is one to one.  Hence, rather than counting solutions to \eqref{eq:thueEq}, we instead count rational solutions to \[\left|f\left(\frac{x}{y}\right)\right| = \frac{1}{y^n}.\]

Thomas, in \cite{Thomas2000}, shows that the number of regular solutions $(p,q)$ of \eqref{eq:thueEq} for which there exists a critical point of $f(x)$, $\tau$, so that $\frac{p}{q}$ belongs to $\tau$ is no larger than $\ell(n)$ (see the completion of the proof of Thomas' Theorem 2.2, given after the statement of Theorem 7.1).  So it only remains to show

\begin{lemma}
	The number of regular solutions, $(p,q)$, of \eqref{eq:thueEq} for which $\frac{p}{q}$ belongs to a real root of $f$ is no larger than $z(n)$.
\end{lemma}

By Theorem 2.2 in \cite{Thomas2000}, it suffices to show Lemma \ref{trinomialRealRootCount} for real roots of $f$ which are greater than 1.  Then by Lemma 2.4 of \cite{Thomas2000}, we conclude that any regular $(p,q)$ for which $\frac{p}{q}$ belongs to an exceptional point greater than 1 has $p > q \geq 1$ and so we may assume that $p > q \geq 1$.  Defining \[p_0(n) := \begin{cases} 3 & \text{if } 6 \leq n \leq 8 \\ 2 & \text{if } n \geq 9 \end{cases},\] we note that any regular solution $(p,q)$ with $p > q \geq 1$ must satisfy
\begin{align} \label{eq:specialSolutionDef}
	p \geq p_0(n)
\end{align}
except for possibly $(2,1)$ when $n \leq 8$.

\begin{definition}
	A solution, $(p,q)$ to equation \eqref{eq:thueEq} with $p > q \geq 1$ and $p \geq p_0(n)$ is called \emph{special}.
\end{definition}

Since at most one solution is not special in the case that $6 \leq n \leq 8$, it suffices to show the following lemma, which will be our final reduction.

\begin{lemma}\label{trinomialRealRootCount}
	Let $\alpha > 1$ be a real root of $f(x)$.  Then the number of special solutions $(p,q)$ of \eqref{eq:thueEq} for which $\frac{p}{q}$ belongs to $\alpha$ is no greater than $z(n) - 1$ if $6 \leq n \leq 8$ and no greater than $z(n)$ if $n \geq 9$.
\end{lemma}

To prove lemma \ref{trinomialRealRootCount}, we split solutions into two cases: small and large.  For $F(x,y)$ of degree $n$ and na\"ive height $H$, we choose a constant $Y_F = H^{\chi_n}\cdot e^{\pi_n}$ (for some values $\chi_n$ and $\pi_n$ to be specified later, but which depend only on $n$) and make the following definition.

\begin{definition}
 	A special solution $(p,q)$ to \eqref{eq:thueEq} is \emph{small} if $q \leq Y_F$ and is \emph{large} otherwise.
\end{definition} 


\subsection{Small Special Solutions}

One of Thomas' main achievements in \cite{Thomas2000} is the following theorem (numbered 4.1 in \cite{Thomas2000}):

\begin{theorem}\label{qInc}
	Suppose that $F(x,y) \in \Z[x,y]$ is an irreducible (over $\Z$) trinomial binary form of degree $n \geq 5$ and na\"ive height $H$.  Let $(p,q)$ and $(p',q')$ be special solutions to \eqref{eq:thueEq} which belong to a real root and suppose $q' > q$.  Then
	\begin{align}
		q' > \frac{H^{d/n}p^{n^*-d}q^d}{K_d(n)}
	\end{align}
	where $n^* := \frac{n-2}{2}$, $d$ is chosen to be any real number satisfying $0 \leq d \leq n^*$, and \[K_d(n) := m_n(r_n(1+u_n))^d\] where
	\begin{align*}
		m_n = 2\sqrt{\frac{2n}{(n-1)(n-2)}}, \qquad r_n = (2.032)^{1/n}, \qquad u_n = \sqrt{\frac{2}{(n-2)p_0^n}}.
	\end{align*}
\end{theorem}

This approximation result will be helpful in proving the following proposition:

\begin{proposition}
	Let $\alpha > 1$ be a real root of $f(x)$.  There are no more than
	\begin{align}
		T:= \left\lfloor\max\left(\frac{\log\left(\frac{\chi_n}{\frac{d_0(d-1)+d}{n(d-1)}} + 1\right)}{\log d}, \frac{\log\left(\frac{\pi_n}{\log K_d(n)^{-\frac{1}{d-1}}Q_1} + 1\right)}{\log d}\right)\right\rfloor + 2\label{eq:TDef}
	\end{align}
	small special solutions $(p,q)$ where $p/q$ belongs to $\alpha$.
\end{proposition}

\begin{proof}
	If there less than 2 special solutions $(p,q)$ where $p/q$ belongs to $\alpha$, then we are done.  Otherwise, suppose that there are exactly $t + 2$ small special solutions $(p,q)$ where $p/q$ belongs to $\alpha$ and $t \geq 0$.  Label those $t + 2$ solutions as $(p_0,q_0),\ldots,(p_{t+1},q_{t+1})$ ordered so that \[1 \leq q_0 < q_1 <\ldots<q_{t+1} \leq Y_F\] (the strict inequality follows from the fact that the $\frac{p_i}{q_i}$ are principal convergents to $\alpha$ by corollary 3.2 in \cite{Thomas2000}).

	Choose numbers $d_0,d \in \R_{>0}$ and make the following definitions:
	\begin{align*}
		c_0 &:= n^*-d_0, & K_0 &:= K_{d_0}(n), & Q_1 &:= \frac{p_0(n)^{c_0}}{K_0}.
	\end{align*}
	In particular, choose $d$ and $d_0$ so that
	\begin{align}
		&0 \leq d_0 \leq n^*-1.4,\nonumber\\
		&1 < d \leq n^*,\nonumber\\
		&Q_1^{d-1} > \max(1,K_d(n)).\label{eq:Q1size}
	\end{align}

	In the proof of Proposition \ref{TZForLargeN} and in the computations in section \ref{paramsForSmallN}, we show by example that choosing such $d$ and $d_0$ are possible.

	First, observe that by Theorem \ref{qInc} applied to $q_1 > q_0 \geq 1$ (and using the observation that $p_0 \geq p_0(n)$), we get $q_1 > H^{d_0/n}Q_1$.

	Here is where we depart from Thomas' method.  We now aim to apply Lemma \ref{GPImprovement} to $H^{d_0/n}Q_1 < q_1 < q_2 < \dots < q_{t+1} \leq Y_F$.  In the notation of Lemma \ref{GPImprovement}, we have $L = H^{d_0/n}Q_1$, $M = Y_F$, $p = d+1$, and $T = \frac{K_d(n)}{H^{d/n}}$ (technically, $T$ should have a factor of $1/p^{n^*-d}$, but the result of Lemma \ref{GPImprovement} still holds if $T$ is replaced by something larger and moreover, we will choose $d = n^*$ later, rendering the difference moot).  To apply the conclusion of Lemma \ref{GPImprovement}, we need to check that $p > 2$ (trivial based on the fact that $d$ is chosen to be greater than 1) and we need to check that $L^{p-2} > T$.  But this occurs if and only if \[\left(H^{d_0/n}Q_1\right)^{d-1} > \frac{K_d(n)}{H^{d/n}},\] i.e. $H^{(d_0(d-1) + d)/n}Q_1^{d-1} > K_d(n)$, which is guaranteed by \eqref{eq:Q1size}.

	Now applying Lemma \ref{GPImprovement} and using the fact that $t$ is an integer yields
	\begin{align*}
		t 	&\leq \Bigg\lfloor\frac{\log\left[\frac{\log\left(Y_F\left(\frac{K_d(n)}{H^{d/n}}\right)^{-\frac{1}{d-1}}\right)}{\log\left(H^{d_0/n}Q_1\left(\frac{K_d(n)}{H^{d/n}}\right)^{-\frac{1}{d-1}}\right)}
			\right]}{\log d}\Bigg\rfloor\\
			&= \bigg\lfloor\frac{\log\left[\frac{\log\left(Y_FK_d(n)^{-\frac{1}{d-1}}H^{\frac{d}{n(d-1)}}\right)}{\log\left(K_d(n)^{-\frac{1}{d-1}}H^{\frac{d_0}{n}+\frac{d}{n(d-1)}}Q_1\right)}\right]}{\log d}\bigg\rfloor\\
			&\leq \bigg\lfloor\frac{\log\left[\frac{\log(Y_F)}{\log\left(K_d(n)^{-\frac{1}{d-1}}H^{\frac{d_0}{n}+\frac{d}{n(d-1)}}Q_1\right)} + 1\right]}{\log d}\bigg\rfloor.
	\end{align*}
	where in the last step, we use the fact that $Q_1 > 1$.  Now using the definition $Y_F = H^{\chi_n} \cdot e^{\pi_n}$, we have
	\begin{align*}
		t 	&\leq \Bigg\lfloor\frac{\log\left[\frac{\chi_n\log H + \pi_n}{\frac{d_0(d-1)+d}{n(d-1)}\log H + \log\left( K_d(n)^{-\frac{1}{d-1}}Q_1\right)} + 1\right]}{\log d}\Bigg\rfloor\\
			&\leq \left\lfloor\max\left(\frac{\log\left(\frac{\chi_n}{\frac{d_0(d-1)+d}{n(d-1)}} + 1\right)}{\log d}, \frac{\log\left(\frac{\pi_n}{\log K_d(n)^{-\frac{1}{d-1}}Q_1} + 1\right)}{\log d}\right)\right\rfloor\\
			&= T - 2.
	\end{align*}
	Therefore, the number of small special solutions $(p,q)$ for which $p/q$ belongs to $\alpha$ is $t + 2 \leq T$.
\end{proof}


\subsection{Large Special Solutions}

Here we follow Thomas in \cite{Thomas2000} as he follows Bombieri-Schmidt in \cite{Bombieri1987}.  If we choose numbers $a$ and $b$ satisfying
\begin{align}\label{eq:UVReq}
	0 < a < b < 1 - \sqrt{2 \cdot \frac{n + a^2}{n^2}}
\end{align}
then we can define
\begin{align*}
	L &= \frac{\sqrt{2(n + a^2)}}{1 - b} & D &= \frac{L}{n - L} & A &= \frac{1}{a^2} & E &= \frac{1}{2(b^2 - a^2)}.
\end{align*}
Now we choose
\begin{align}
	\chi_n 	&= D(A + 1) + 1\label{eq:chiDef}\\
	\pi_n 	&= (D(4 + A) + 2)\log(2) + \frac{(D + 1)\log(n)}{2} + \frac{nAD}{2}.\label{eq:piDef}
\end{align}

With these choices of $\pi_n$ and $\chi_n$, we aim to apply Lemma 2 of \cite{Bombieri1987} and conclude the following:

\begin{proposition}\label{largeSpecialSolutionBound}
	Suppose $\alpha > 1$ is a real root of $f(x)$.  If $\chi_n \geq 2$ and $\pi_n \geq 5\log(2) + 2\log(n)$, then there are at most
	\begin{align}\label{eq:ZDef}
		Z :=\left\lfloor\frac{\log E + 2\log(n) - \log(L - 2)}{\log(n-1)}\right\rfloor + 2
	\end{align}
	large special solutions belonging to $\alpha$.
\end{proposition}

The proof of this proposition relies on the two following lemmas:

\begin{lemma}\label{largeYF}
	$Y_F$ as defined here is greater than or equal to $Y_0$ as defined in \cite{Bombieri1987}.
\end{lemma}

Lemma \ref{largeYF} ensures that any large solution in the sense of this paper is a large solution in the sense of Bombieri and Schmidt.

\begin{lemma}\label{belongsToImpliesClosest}
	Suppose that $\chi_n \geq 2$ and $\pi_n \geq 5\log(2) + 2\log(n)$.  If $\alpha > 1$ is a real root of $f(x)$ and $(p,q)$ is a large special solution of \eqref{eq:thueEq} so that $p/q$ belongs to $\alpha$, then $\alpha$ is the closest (complex) root of $f(x)$ to $p/q$.
\end{lemma}

Given an algebraic $\beta$, Lemma 2 of \cite{Bombieri1987} only counts the number of rational numbers which are nearest to $\beta$ out of all of the conjugates of $\beta$ and which form good approximations of $\beta$.  If there were a real root $\alpha > 1$ of $f(x)$ and a large special solution $(p,q)$ of \eqref{eq:thueEq} for which $p/q$ belonged to $\alpha$ yet there was a root $\beta$ of $f(x)$ with $\beta$ closer to $p/q$ than $\alpha$, Lemma 2 of \cite{Bombieri1987} would not count $p/q$.  However, Lemma \ref{belongsToImpliesClosest} confirms that this is not the case.

We first prove these two lemmas:

\begin{proof}[Proof of Lemma \ref{largeYF}]
	$Y_0$ depends on the Mahler measure $M(F)$ rather than the height $H(F)$.  These are related (for trinomials $F(x,y)$) by $M(F) \leq 3^{1/2}H(F)$, which follows from the fact that $M(F) \leq \ell_2(F)$ (see Lemma 1.6.7 in \cite{BombieriEnrico2001Hidg}).  Now, using Thomas' notation, we have that \[Y_0 := \left(2C\right)^{\frac{1}{n-\lambda}}\left(4e^{A_1}\right)^{\frac{\lambda}{n - \lambda}}\] where
	\begin{align*}
		C 	&= (2n^{1/2} M(F))^n\\
		t 	&= \sqrt{\frac{2}{n + a^2}}\\
		A_1 &= \frac{t^2}{2 - nt^2}\left(\log M(F) + \frac{n}{2}\right)\\
		\lambda &= \frac{2}{t(1 - b)}.
	\end{align*}

	Some of our other constants regularly appear in the estimation which shows $Y_0 < Y_F$ and we list them here for simplicity:

	\begin{align*}
		&A = \frac{1}{a^2} =  \frac{2}{2(n + a^2) - 2n} = \frac{2}{n + a^2}\cdot\frac{1}{2-\frac{2n}{n + a^2}} = \frac{t^2}{2 - nt^2}\\
		&L = \frac{\sqrt{2(n + a^2)}}{1-b} = \frac{2}{1 - b} \sqrt{\frac{n + a^2}{2}} = \frac{2}{t(1 - b)} = \lambda \\
		&D = \frac{L}{n - L} = \frac{\lambda}{n - \lambda}
	\end{align*}
	Note also that this implies that $D + 1 = \frac{n}{n - \lambda}$.

	Before making the final estimate, we take a moment to observe that \[\left(\sqrt{3}\right)^{\frac{n}{n-\lambda} + AD} < 2^{\frac{n-1}{n-\lambda}}.\]  This estimate is tedious, but not difficult.  One can show that \[\left(\sqrt{3}\right)^{\frac{n}{n-\lambda} + AD} < 2^{\frac{n-1}{n-\lambda}}\] occurs if and only if \[2 < \left(\frac{2}{\sqrt{3}}\right)^{n + A\lambda}.\]  Estimating $A$ from below by \[A > \frac{1}{(1 - \sqrt{2(n+1)/n^2})^2}\] and estimating $\lambda$ from below by $\lambda > \sqrt{2n}$ gives that \[2 < \left(\frac{2}{\sqrt{3}}\right)^{n + A\lambda}\] is implied by \[2 < \left(\frac{2}{\sqrt{3}}\right)^{n + \frac{n^2\sqrt{2n}}{n^2-2n\sqrt{2n + 2} + 2n + 2}}.\]  Upon observing that \[n + \frac{n^2\sqrt{2n}}{n^2-2n\sqrt{2n + 2} + 2n + 2} \geq 20\] when $n \geq 6$ for instance, one can now see that \[2 < \left(\frac{2}{\sqrt{3}}\right)^{n + \frac{n^2\sqrt{2n}}{n^2-2n\sqrt{2n + 2} + 2n + 2}}\] and as a result, we must have \[\left(\sqrt{3}\right)^{\frac{n}{n-\lambda} + AD} < 2^{\frac{n-1}{n-\lambda}}.\]
	We can now conclude
	\begin{align*}
		Y_0 &= (2C)^{\frac{1}{n-\lambda}}(4e^{A_1})^{\frac{\lambda}{n - \lambda}}\\
			&= (2(2n^{1/2}M(F))^n)^{\frac{1}{n - \lambda}}(4e^{A(\log M(F) + \frac{n}{2})})^{\frac{\lambda}{n - \lambda}}\\
			&= 2^{\frac{1 + n + 2\lambda}{n-\lambda}}\cdot n^{\frac{n}{2(n - \lambda)}}\cdot M(F)^{\frac{n}{n - \lambda}}e^{AD(\log M(F) + \frac{n}{2})}\\
			&= 2^{\frac{1 + n + 2\lambda}{n-\lambda}}\cdot n^{\frac{n}{2(n - \lambda)}}\cdot M(F)^{\frac{n}{n - \lambda} + AD}\cdot e^{\frac{ADn}{2}}\\
			&\leq 2^{\frac{1 + n + 2\lambda}{n-\lambda}}\cdot n^{\frac{n}{2(n - \lambda)}}\cdot (\sqrt{3}H(F))^{\frac{n}{n - \lambda} + AD}\cdot e^{\frac{ADn}{2}}.
	\end{align*}
	Next we use the fact that $\left(\sqrt{3}\right)^{\frac{n}{n-\lambda} + AD} < 2^{\frac{n-1}{n-\lambda}}$ to find that
	\begin{align*}
		Y_0	&< 2^{\frac{1 + n + 2\lambda}{n-\lambda} + \frac{n - 1}{n - \lambda} + AD}\cdot n^{\frac{n}{2(n - \lambda)}}\cdot H(F)^{\frac{n}{n - \lambda} + AD}\cdot e^{\frac{ADn}{2}}\\
			&= 2^{\frac{2n + 2\lambda}{n - \lambda} + AD} \cdot n^{\frac{D + 1}{2}} \cdot H(F)^{1 + D + AD} \cdot e^{\frac{ADn}{2}}\\
			&= H(F)^{\chi_n} \cdot \exp\left(\left(\frac{2n + 2\lambda}{n - \lambda} + AD\right)\log(2) + \frac{D + 1}{2}\log n + \frac{ADn}{2}\right)\\
			&= H(F)^{\chi_n} \cdot \exp\left(\left(\frac{4\lambda + 2(n - \lambda)}{n - \lambda} + AD\right)\log(2) + \frac{D + 1}{2}\log n + \frac{ADn}{2}\right)\\
			&= H(F)^{\chi_n} \cdot \exp\left(\left(4D + 2 + AD\right)\log(2) + \frac{D + 1}{2}\log n + \frac{ADn}{2}\right)\\
			&= H(F)^{\chi_n} \cdot e^{\pi_n}\\
			&= Y_F.
	\end{align*}
\end{proof}

\begin{proof}[Proof of Lemma \ref{belongsToImpliesClosest}]
	Since $\frac{p}{q}$ is a large special solution, we have
	\begin{align*}
		p > q \geq Y_F &= H^{\chi_n}e^{\pi_n}.
	\end{align*}
	Since $\frac{p}{q}$ belongs to $\alpha$, Thomas' Corollary 3.1 in \cite{Thomas2000} indicates that
	\begin{align*}
		\left|\frac{p}{q} - \alpha\right| 	&< \frac{1}{p^{n^*}q}\\
											&< \frac{1}{Y_F^{n/2}}.
	\end{align*}

	Suppose, by contradiction, that there exists $\beta \in \C$ with $f(\beta) = 0$ and $\left|\frac{p}{q} - \beta\right| < \left|\frac{p}{q} - \alpha\right|$.  Then by the triangle equality, we find that
	\begin{align}
		\left|\alpha - \beta\right| &\leq \left|\frac{p}{q} - \beta\right| + \left|\frac{p}{q} - \alpha\right|\nonumber\\
									&< \frac{2}{Y_F^{n/2}}\label{eq:distBetweenRoots}
	\end{align}
	Since $\alpha$ and $\beta$ are distinct roots of $f$, Theorem 4 in \cite{Rump1979} indicates that
	\begin{align}
		|\alpha - \beta| > \frac{1}{2n^{n/2 + 2} (4H)^n} = \frac{1}{2^{2n + 1}n^{n/2 + 2}H^n}\label{eq:RumpRootSep}.
	\end{align}
	Combining \eqref{eq:distBetweenRoots} and \eqref{eq:RumpRootSep}, we find that \[\frac{1}{2^{2n + 1}n^{n/2 + 2}H^n} < \frac{2}{Y_F^{n/2}}\] and rearranging yields \[\frac{Y_F^{n/2}}{2^{2n + 2}n^{n/2 + 2}H^n} <  1.\]
	From here, we can use the fact that $Y_F = H^{\chi_n}e^{\pi_n}$ to find
	\begin{align*}
		1 	&> \frac{H^{n(\chi_n/2 - 1)}e^{n\pi_n/2}}{2^{2n + 2}n^{n/2 + 2}}\\
			&\geq \frac{e^{n\pi_n/2}}{2^{2n + 2}n^{n/2 + 2}}
	\end{align*}
	where the last inequality follows because $\chi_n \geq 2$.  After rearranging, this implies that
	\begin{align*}
		\pi_n 	&< \frac{(4n + 4)\log(2) + (n + 4)\log(n)}{n}\\
				&\leq 5\log(2) + 2\log(n)
	\end{align*}
	where the last inequality follows from the fact that $n \geq 6$.  However, the last inequality contradicts our hypothesis that $\pi_n \geq 5\log(2) + 2\log(n)$, so no such $\beta$ can exist and the closest root of $f(x)$ to $p/q$ is $\alpha$.
\end{proof}

Finally, we prove proposition \ref{largeSpecialSolutionBound}.

\begin{proof}[Proof of Proposition \ref{largeSpecialSolutionBound}.]
	Let $\alpha > 1$ be a real root of $f(x)$.  By lemma \ref{largeYF}, every large special solution $(p,q)$ so that $p/q$ belongs to $\alpha$ is large in the sense of \cite{Bombieri1987}.  Moreover, by lemma \ref{belongsToImpliesClosest}, any large special solution $(p,q)$ so that $p/q$ belongs to $\alpha$ has \[\left|\frac{p}{q} - \alpha\right| = \min_{f(\beta) = 0} \left|\frac{p}{q} - \beta\right|.\]  Hence, every large special solution $(p,q)$ so that $p/q$ belongs to $\alpha$ is large (in the sense of \cite{Bombieri1987}) and is nearest to $\alpha$ among all the roots of $f(X)$.  Lemma 2 of \cite{Bombieri1987} indicates that there are no more than \[Z = \left\lfloor\frac{\log E + 2\log(n) - \log(L - 2)}{\log(n-1)}\right\rfloor + 2\] large solutions $(p,q)$ so that $p/q$ is nearest to $\alpha$ among all the roots of $f(X)$ and so we conclude that there are no more than $Z$ large special solutions $(p,q)$ with $p/q$ belonging to $\alpha$.
\end{proof}


\subsection{Choosing Parameters for Large Degrees}

Begin by assuming $n \geq 507$.  We handle all smaller instances of $n$ computationally.

\begin{proposition}\label{TZForLargeN}
	For $n \geq 507,$ we can take $d_0 = \frac{n^*}{2}$, $d = n^*$, $a = \frac{1}{4}$, $C = 7/6$, $c = \frac{8}{9C^2 - 1}$, $b = 1 - \frac{\sqrt{2n + \frac{1}{8}}}{\frac{cn^2}{n-1} + 2}$ and obtain $T = 2$ and $Z = 2$.
\end{proposition}

Observe first that these are the smallest possible values of $T$ and $Z$.

\begin{proof}
	To show this, we first must show that these choices of $d_0, d, a,$ and $b$ are valid.

	Certainly $0 \leq d_0 \leq n^* - 1.4$ and $1 < d \leq n^*$.  All that remains to show for $d_0$ and $d$ is \eqref{eq:Q1size}.  We have
	\begin{align*}
		Q_1^{d - 1}	&= \left(\frac{p_0^{c_0}}{K_0}\right)^{d-1}\\
					&\geq \left(\frac{2^{n^*/2}}{K_{n^*/2}(n)}\right)^{\frac{n^*}{2} - 1}.
	\end{align*}
	But observe that
	\begin{align*}
		K_{n^*/2}(n) 	&= 2\sqrt{\frac{2n}{(n-1)(n-2)}}\left(2.032^{1/n}\left(1 + \sqrt{\frac{2}{(n-2)p_0^n}}\right)\right)^{n^*/2}\\
						&\leq 2 \cdot 2.032\left(1 + \frac{1}{\sqrt{(n-2)2^{n-1}}}\right)^{\frac{n-2}{4}}\\
						&\leq 5\left(1 + \frac{1}{n-2}\right)^{\frac{n-2}{4}}\\
						&\leq 5e^{1/4}
	\end{align*}
	so $Q_1^{d-1}$ is certainly greater than 1.  Similar reasoning shows that $K_d(n) \leq 5e^{1/2}$, so it is certainly also the case that \[K_d(n) \leq 5e^{1/2} < \left(\frac{2^{n^*/2}}{5e^{1/4}}\right)^{\frac{n^*}{2} - 1} \leq \left(\frac{2^{n^*/2}}{K_{n^*/2}(n)}\right)^{\frac{n^*}{2} - 1} = Q_1^{b-1}.\]  Hence, our choices of $d$ and $d_0$ are valid.

	Next, we wish to check that our choices for $a$ and $b$ are valid.  To check $0 < a < b$, note that
	\begin{align}
		b 	&= 1 - \frac{\sqrt{2n + \frac{1}{8}}}{\frac{cn^2}{n-1} + 2} \geq 1 - \frac{2\sqrt{n}}{\frac{cn^2}{n-1}} = 1 - \frac{2(n-1)\sqrt{n}}{cn^2} \geq 1 - \frac{2}{c\sqrt{n}}\nonumber\\
			&\geq 1 - \frac{2}{c\sqrt{507}}> \frac{1}{4} = a.\label{vLowerBound}
	\end{align}
	To check that $b < 1 - \frac{\sqrt{2n + 2a^2}}{n}$, it suffices to show that $n > \frac{cn^2}{n-1} + 2$.  But this occurs if and only if $(1-c)n^2 - 3n + 2 > 0$, i.e. $n > \frac{3 + \sqrt{9 - 8(1-c)}}{2(1-c)} \approx 9.66$, which we certainly have.

	To show that $T = 2$, we claim that we have the following two inequalities (and from equation \eqref{eq:TDef}, it will follow that $T = 2$):
	\begin{align}
		\frac{\chi_n}{\frac{d_0(d-1)+d}{n(d-1)}} + 1 < d \label{eq:ChiReq}\\
		\frac{\pi_n}{\log K_d(n)^{-\frac{1}{d-1}}Q_1} + 1 < d \label{eq:PiReq}
	\end{align}

	We first show \eqref{eq:ChiReq}.  Substituting $d_0 = \frac{n^*}{2}$ and $d = n^*$, observe that \eqref{eq:ChiReq} is equivalent to \[\chi_n < \frac{\left(\frac{n-2}{4}\right)\left(\frac{n-4}{2} - 1\right) + \frac{n-2}{2}}{n} = \frac{n-2}{8} = \frac{n^*}{4}.\]
	Keeping an eye on the definition of $\chi_n$ given in equation \eqref{eq:chiDef}, we have that
	\begin{align*}
		A &= 16\\
		C &= \frac{7}{6}\\
		c &= \frac{8}{9C^2 - 1} = \frac{32}{45}\\
		b &= 1 - \frac{\sqrt{2n + \frac{1}{8}}}{\left(\frac{cn^2}{n - 1} + 2\right)}.
	\end{align*}
	All of these together yield
	\begin{align}\label{eq:LDef}
		L = \frac{cn^2}{n-1} + 2 = \frac{32}{45}\left(n + 1 + \frac{1}{n - 1}\right) + 2
	\end{align}
	and it is now easy to check that \[\frac{32}{45}n \leq L \leq \frac{32}{45}n + 3.\]

	From here we have
	\begin{align*}
		D 		&= \frac{L}{n - L} \leq \frac{\frac{32}{45}n + 3}{n - \left(\frac{32}{45}n + 3\right)} = \frac{\frac{32}{45}n + 3}{\frac{13}{45}n - 3} = \frac{32}{13} + \frac{6075}{13(13n - 135)} \leq 2.54 \\
		D 		&\geq \frac{\frac{32}{45n}}{n - \frac{32}{45}n} = \frac{32}{13} \approx 2.46
	\end{align*}
	when we use the fact that $n \geq 507$.  To convert these into estimates on $\chi_n$, we have
	\begin{align}
		\chi_n 	&= 17D + 1 \leq \frac{94853}{2152} \leq 44.08\label{eq:chiNUpperBound}\\
		\chi_n 	&= 17D + 1 \geq \frac{557}{13} \geq 42.8.\label{eq:chiNLowerBound}
	\end{align}	
	Since $n \geq 507$, we now have $\chi_n \leq 44.08 < \frac{n-2}{8}$ which confirms equation \eqref{eq:ChiReq}.

	Equation \eqref{eq:PiReq} is more complicated to handle.  Observe that by equation \eqref{eq:piDef}, we have
	\begin{align}\label{eq:piNUpperBound}
		\pi_n 	&= (D(4 + A) + 2)\log 2 + \frac{(D + 1)\log n}{2} + \frac{ADn}{2}\nonumber\\
			&\leq 36.6 + 1.77\log n + 20.28n\nonumber\\
			&\leq 37 + 21n.
	\end{align}
	For reference later, we will also note
	\begin{align}\label{eq:piNLowerBound}
		\pi_n 	&\geq 35.5 + 1.7\log n + 19.6n\nonumber\\
				&\geq 46 + 19n.
	\end{align}
	It will additionally be helpful for us to have an estimate on $K_d(n)$.  We have
	\begin{align*}
		K_d(n) 	&= 2\sqrt{\frac{2n}{(n-1)(n-2)}}\left(2.032^{1/n}\left(1+\sqrt{\frac{2}{(n-2)p_0^n}}\right)\right)^d\\
				&\leq 4\sqrt{\frac{3n-3}{(n-1)(n-2)}}\left(1+\sqrt{\frac{2n-4}{(n-2)p_0^n}}\right)^{d}\\
				&\leq 4\sqrt{\frac{3}{n-2}}\left(1+\sqrt{\frac{2}{p_0^n}}\right)^{n^*}.
	\end{align*}
	Now, since $p_0 \geq 2$, we have
	\begin{align}\label{eq:KbUpperBound}
		K_d(n)	&\leq 4\sqrt{\frac{3}{n-2}}\left(1+\sqrt{\frac{2}{2^n}}\right)^{\frac{n-2}{2}}
				\leq 4\sqrt{\frac{3}{n-2}}\left(1+\frac{1}{2^{\frac{n-2}{2}}}\right)^{\frac{n-2}{2}}\nonumber\\
				&\leq 4\sqrt{\frac{3}{n-2}}\left(1 + \frac{1}{\left(\frac{n-2}{2}\right)}\right)^{\frac{n-2}{2}} 
				\leq 4e\sqrt{\frac{3}{n-2}}
	\end{align}
	and similar reasoning yields
	\begin{align}\label{eq:KbzeroUpperBound}
		K_{d_0}(n) \leq 4\sqrt{\frac{3}{n-2}}\left(1 + \frac{1}{2^{n^*}}\right)^{n^*/2}.
	\end{align}
	
	Combining the upper bounds in equations \eqref{eq:piNUpperBound}, \eqref{eq:KbUpperBound}, and \eqref{eq:KbzeroUpperBound} with the fact that for $n \geq 270$, \[\pi_n < 37+21n < \frac{\log(1.9)}{8}(n-2)(n-4)\] yields
	\begin{align*}
		(d-1)\log\left[K_d(n)^{-\frac{1}{d-1}}Q_1\right]
		&= \log\left[K_d(n)^{-1}Q_1^{d-1}\right]\\
		&\geq \log\left[\frac{1}{4} \cdot \inv{e} \cdot \left(\frac{n-2}{3}\right)^{1/2} \cdot \left(\frac{p_0^{c_0}}{K_0}\right)^{n^*-1}\right]\\
		&\geq \log\left[\frac{1}{4} \cdot \inv{e} \cdot \left(\frac{n-2}{3}\right)^{1/2}\left(\frac{2^{n^*/2}}{K_{d_0}(n)}\right)^{n^*-1}\right]\\
		&\geq \log\left[\frac{1}{4} \cdot \inv{e} \cdot \left(\frac{n-2}{3}\right)^{1/2}\left(\frac{2^{n^*/2}}{\left(1+2^{-n^*}\right)^{n^*/2}}\right)^{\frac{n-4}{2}}\left(\frac{1}{4}\sqrt{\frac{n-2}{3}}\right)^{\frac{n-4}{2}}\right]\\
		&\geq \log\left[\left(\frac{1}{4}\right)^{\frac{n-2}{2}}\left(\frac{n-2}{3}\right)^{\frac{n-2}{4}}\inv{e}\left(\frac{2}{1+2^{-n^*}}\right)^{\left(\frac{n-2}{4}\right)\left(\frac{n-4}{2}\right)}\right]\\
		&\geq \log\left[\left(\frac{n-2}{48}\right)^{\frac{n-2}{4}}\inv{e}\cdot 1.9^{\frac{(n-2)(n-4)}{8}}\right]\\
		&\geq \frac{\log1.9}{8}(n-2)(n-4) + \frac{n-2}{4}\log\left(\frac{n-2}{48}\right) - 1\\
		&\geq \frac{\log(1.9)}{8}(n-2)(n-4)\\
		&> \pi_n
	\end{align*}
	which now implies that equation \eqref{eq:PiReq} is satisfied.  Hence, we conclude that $T = 2$.

	Finally, we check that $Z = 2$.  In order to use Proposition \ref{largeSpecialSolutionBound}, we must verify that  $\chi_n \geq 2$ and $\pi_n \geq 5\log(2) + 2\log(n)$.  However, these quickly follow from \eqref{eq:chiNLowerBound} and \eqref{eq:piNLowerBound}.

	As before in equation \eqref{vLowerBound}, we have $b \geq 1 - \frac{2}{c\sqrt{507}} > 0.87509$, so
	\begin{align*}
		E = \frac{1}{2(b^2 - a^2)} < 0.711 < c
	\end{align*}
	and so (also using \eqref{eq:LDef})
	\begin{align*}
		\frac{\log E + 2\log(n) - \log(L - 2)}{\log(n-1)} < \frac{\log\left(\frac{cn^2}{L-2}\right)}{\log(n-1)} = 1.
	\end{align*}
	Therefore, by \eqref{eq:ZDef}, we note that $Z = 2$.
\end{proof}

Note that Proposition \ref{TZForLargeN} proves Lemma \ref{trinomialRealRootCount} for $n \geq 507$.


\subsection{Choosing Parameters for Small Degrees}\label{paramsForSmallN}

For $n \leq 506$, we make parameter choices listed in the Jupyter notebook listed here:

\begin{center}
\url{https://pages.uoregon.edu/gknapp4/files/trinomial_computations.ipynb}
\end{center}

One can check (and the code checks this automatically) that the parameter choices satisfy \eqref{eq:Q1size}, \eqref{eq:UVReq} along with the necessary bounds on $\pi_n$ and $\chi_n$ in order to use proposition \ref{largeSpecialSolutionBound}, and yield the $T$ and $Z$ values giving $z(n) = T + Z + 1$ when $6 \leq n \leq 8$ and $z(n) = T + Z$ for $n \geq 9$.  Therefore, these computations verify Lemma \ref{trinomialRealRootCount} for $n \leq 506$, which concludes our investigation.

In brief, the code picks a value of $n$, sets $d = n^*$, brute force loops over a number of valid values for the parameters $d_0,a,b$, computes the corresponding $T$ and $Z$ values defined in equations \eqref{eq:TDef} and \eqref{eq:ZDef}, and records the values of $d_0,a,$ and $b$ which minimize $T+Z$.  The exact code can be found in appendix \ref{pythonCode} and the following table contains a summary of some of the more interesting data points listed in the Jupyter notebook.

\begin{longtable}{*7{|c|}}
\hline
	$n$ &   $d_0$ &   $d$ &      $a$ &      $b$ &   $T$ &   $Z$ \\
\hline
\endhead
	6 &   0     &   2   &  0.18 &  0.29 &    10 &     4 \\
	7 &   0.55  &   2.5 &  0.2  &  0.28 &     7 &     4 \\
	8 &   0.992 &   3   &  0.16 &  0.41 &     7 &     3 \\
	9 &   0.882 &   3.5 &  0.17 &  0.4  &     6 &     3 \\
	10 &   1.196 &   4   &  0.23 &  0.41 &     5 &     3 \\
	11 &   1.705 &   4.5 &  0.14 &  0.37 &     5 &     3 \\
	12 &   2.088 &   5   &  0.27 &  0.41 &     4 &     3 \\
	13 &   2.255 &   5.5 &  0.2  &  0.37 &     4 &     3 \\
	14 &   2.53  &   6   &  0.16 &  0.35 &     4 &     3 \\
	15 &   3.009 &   6.5 &  0.13 &  0.34 &     4 &     3 \\
	16 &   3.136 &   7   &  0.11 &  0.32 &     4 &     3 \\
	17 &   3.965 &   7.5 &  0.33 &  0.43 &     3 &     3 \\
	18 &   4.29  &   8   &  0.27 &  0.39 &     3 &     3 \\
	$\vdots$ & $\vdots$ & $\vdots$ & $\vdots$ & $\vdots$ & $\vdots$ & $\vdots$ \\
	37 &   7.728 &  17.5 &  0.08 &  0.25 &     3 &     3 \\
	38 &   9.794 &  18   &  0.07 &  0.25 &     3 &     3 \\
	39 &  11.97  &  18.5 &  0.43 &  0.47 &     2 &     3 \\
	40 &  12.144 &  19   &  0.41 &  0.46 &     2 &     3 \\
	$\vdots$ & $\vdots$ & $\vdots$ & $\vdots$ & $\vdots$ & $\vdots$ & $\vdots$ \\
	218 &  57.564 & 108   &  0.03 &  0.87 &     3 &     2 \\
	219 & 68.2227 & 108.5 & 0.399258 & 0.883258 &     2 &     2 \\
	220 & 68.6488 & 109   & 0.408477 & 0.884477 &     2 &     2 \\
	$\vdots$ & $\vdots$ & $\vdots$ & $\vdots$ & $\vdots$ & $\vdots$ & $\vdots$ \\
	506 & 218.022 & 252   & 0.517076 & 0.927076 &     2 &     2 \\
	507 & 218.457 & 252.5 & 0.517138 & 0.927138 &     2 &     2 \\
\hline
\end{longtable}

\subsection{Attaining Bounds With Examples}

Theorem \ref{totalTrinomialSolutions} indicates that for $n \geq 219$, there are no more than 40 distinct solutions to equation \eqref{eq:thueEq} when $F(x,y)$ is a trinomial.  For smaller $n$, this upper bound is even larger.  Of interest is whether or not it is possible to find a particular trinomial for which \eqref{eq:thueEq} has 40 distinct solutions.  The computer algebra system GP has a method called \verb thue \,which, on input a Thue equation, will output the solutions to that Thue equation\footnote{While the accuracy of GP's thue method relies on the Generalized Riemann Hypothesis to solve the Thue equation $F(x,y) = h$, our use of it does not because thue does not assume GRH when solving the specific equation $F(x,y) = \pm 1$.}.  The author has used this method to create a method in Sage which, on input a degree $n$ and height $H$, will compute the solutions to every trinomial Thue equation of degree $n$ and height $H$.  The method can be found in appendix \ref{sageMethod} and the raw data can be found on the author's website at:

\begin{center}
\url{https://pages.uoregon.edu/gknapp4/files/trinomial_solution_data.zip}
\end{center}

The maximal number of solutions to equation \eqref{eq:thueEq} for a trinomial $F(x,y)$ of degree $n$ and height $H$ are found in the table below.  Notably, no trinomial has been found with more than 8 solutions to \eqref{eq:thueEq}, which is far from the upper bound of 40.  A hyphen in the table means that the case in question has not yet been computed.

\begin{table}[h]
	\centering
	\begin{tabular}{|c | c | c | c | c | c | c | c | c | c | c | c|}\hline
		 		& $H = 1$ & $H = 2$ & $H = 3$ & $H = 4$ & $H = 5$ & $H = 6$ & $H = 7$ & $H = 8$ & $H = 9$ & $H = 10$ & $H = 11$\\\hline
		$n = 6$ & $8$ & $6$ & $8$ & $8$ & $6$ & $6$ & $6$ & $6$ & $8$ & $6$ & $6$\\\hline
		$n = 7$ & $8$ & $6$ & $8$ & $8$ & $6$ & $6$ & $6$ & $6$ & $8$ & $6$ & $6$\\\hline
		$n = 8$ & $8$ & $6$ & $8$ & $8$ & $6$ & $6$ & $6$ & $6$ & $8$ & $6$ & $6$\\\hline
		$n = 9$ & $8$ & $6$ & $8$ & $8$ & $6$ & $6$ & $6$ & $6$ & $8$ & $6$ & $6$\\\hline
		$n = 10$ & $8$ & $6$ & $8$ & $8$ & $6$ & $6$ & $6$ & $6$ & $8$ & - & - \\\hline
		$n = 11$ & $8$ & $6$ & $8$ & $8$ & $6$ & $6$ & $6$ & $6$ & $8$ & - & - \\\hline
		$n = 12$ & $8$ & $6$ & $8$ & $8$ & $6$ & $6$ & $6$ & - & - & - & - \\\hline
		$n = 13$ & $8$ & $6$ & $8$ & $8$ & $6$ & $6$ & - & - & - & - & - \\\hline
		$n = 14$ & $8$ & $6$ & $8$ & $8$ & $6$ & $6$ & - & - & - & - & - \\\hline
		$n = 15$ & $8$ & $6$ & $8$ & - & - & - & - & - & - & - & -\\\hline
	\end{tabular}
	\caption{The maximal number of solutions to equation \eqref{eq:thueEq} for a trinomial $F(x,y)$ of degree $n$ and height $H$.}
\end{table}

\section{Acknowledgments}

I would like to thank Shabnam Akhtari for her guidance on this project and paper, Ben Young for his suggestions on the computations and feedback on my code, and Cathy Hsu for her mentorship and general research advice.  This research was supported by NSF grant number 2001281.

\bibliographystyle{plain}

\bibliography{library}

\newpage

\appendix

\section{Python Code Minimizing $T + Z$} \label{pythonCode}

Here are the sequence of methods which minimize the value of $T + Z$ for a given parameter $n$.  For brevity, the version here does not include the clarifying comments which are present in the actual Jupyter notebook.

\begin{lstlisting}[language = Python]
from math import sqrt,floor,log
from collections import deque
import decimal
decimal.getcontext().prec = 10**4

def pZero(n):
    if (n <= 8):
        return 3.0
    else:
        return 2.0

def K(d,n):
    return 2*sqrt(float((2*n)/((n-1)*(n-2))))*(2.032**(1/n)*(1+sqrt(float(2/((n-2)* pZero(n)**n)))))**d

def star(n):
    return (n-2) * 0.5

def qOne(dZero,d,n):
    return pZero(n)**(star(n)-dZero)/K(dZero,n)

def validSmall(dZero,d,n):
    return (0<=dZero) and (dZero <= star(n)-1.4) and (1 < d) and (d <= star(n)) and (qOne(dZero,d,n)>max(1,K(d,n)**(1/(d-1))))

def validLarge(a,b,n):
    return (0<a) and (a < b) and (b < 1 - sqrt(float(2*(n + a**2)/n**2)))

def L(a,b,n):
    return sqrt(float(2*(n+a**2)))/(1-b)

def D(a,b,n):
    return L(a,b,n)/(n - L(a,b,n))

def A(a,b,n):
    return 1/a**2

def chiN(a,b,n):
    return D(a,b,n)*(A(a,b,n) + 1)+1

def piN(a,b,n):
    return (D(a,b,n)*(4 + A(a,b,n)) + 2)*float(log(2))+(D(a,b,n) + 1)*float(log(n))/2 + n*A(a,b,n)*D(a,b,n)/2

def E(a,b,n):
    return 1/(2*(b**2-a**2))

def Z(dZero,d,a,b,n):
    return floor((float(log(E(a,b,n))) + 2*float(log(n)) - float(log(L(a,b,n) - 2)))/float(log(n-1)))+2

def T(dZero,d,a,b,n):
    firstQuantity = float(log(chiN(a,b,n) * n * (d - 1) / (dZero * (d-1) + d) + 1))/float(log(d))
    secondQuantity = float(log(piN(a,b,n)/float(log(K(d,n)**(-1/(d-1))*qOne(dZero,d,n))) + 1))/float(log(d))
    return floor(max(firstQuantity,secondQuantity))+2

def minTPlusZ(nMin,nMax,prec):
    toReturn = []
    for n in range(nMin,nMax):
        nStar = star(n)
        aUpper = (2*n**2 - sqrt(4*n**4 - 4*(n**2 - 2*n)*(n**2 - 2)))/(2*(n**2 - 2))
        a = prec
        tempMinA=prec
        dZero = 0
        tempMinDZero = 0
        b = a + prec
        tempMinB = a + prec
        tempMinT = T(dZero,nStar,a,b,n)
        tempMinZ = Z(dZero,nStar,a,b,n)
        tempMinSum = tempMinT + tempMinZ
        while ((a <= aUpper) and (tempMinSum > 4)):
            while ((dZero <= nStar - 1.4) and (tempMinSum > 4)):
                while ((b < 1-sqrt(2*(n + a**2)/n**2)) and (tempMinSum > 4)):
                    tempT = T(dZero,nStar,a,b,n)
                    tempZ = Z(dZero,nStar,a,b,n)
                    if (tempT + tempZ < tempMinSum):
                        tempMinA = a
                        tempMinB = b
                        tempMinDZero = dZero
                        tempMinT = tempT
                        tempMinZ = tempZ
                        tempMinSum = tempT + tempZ
                    b += prec
                dZero += prec * (nStar - 1.4)
                b = a + prec
            a += prec
            dZero = 0
            b = a + prec
        assert validSmall(tempMinDZero,nStar,n), "d0,d,n are invalid"
        assert validLarge(tempMinA,tempMinB,n), "a,b,n are invalid"
        assert chiN(tempMinA,tempMinB,n) >= 2, "chiN is too small"
        assert piN(tempMinA,tempMinB,n) >= 5*log(2) + 2*log(n), "piN is too small"
        toReturn.append([n,tempMinDZero,nStar,tempMinA,tempMinB,tempMinT,tempMinZ])
    return toReturn

def minNWithMinTPlusZ(nMax,prec):
    n = nMax
    nStar = star(n)
    aUpper = (2*n**2 - sqrt(4*n**4 - 4*(n**2 - 2*n)*(n**2 - 2)))/(2*(n**2 - 2))
    a = aUpper - prec
    b = aUpper
    dZero = nStar - 1.4
    tempTPlusZ = 4
    listOfParams = deque([])
    while ((n >= 6) and (tempTPlusZ == 4)):
        tempT = T(dZero,nStar,a,b,n)
        tempZ = Z(dZero,nStar,a,b,n)
        tempTPlusZ = tempT + tempZ
        while ((a > 0) and (tempTPlusZ > 4)):
            while ((b > a) and (tempTPlusZ > 4)):
                while ((dZero >= 0) and (tempTPlusZ > 4)):
                    tempT = T(dZero,nStar,a,b,n)
                    tempZ = Z(dZero,nStar,a,b,n)
                    tempTPlusZ = tempT + tempZ
                    if (tempTPlusZ == 4):
                        assert validSmall(dZero,nStar,n), "d0,d,n are invalid"
                        assert validLarge(a,b,n), "a,b,n are invalid"
                        assert chiN(a,b,n) >= 2, "chiN is too small"
                        assert piN(a,b,n) >= 5*log(2) + 2*log(n), "piN is too small"
                        listOfParams.appendleft([n,dZero,nStar,a,b,tempT,tempZ])
                    dZero -= prec*(nStar - 1.4)
                b -= prec
                dZero = nStar - 1.4
            a -= prec
            b = aUpper
        n -= 1
        aUpper = (2*n**2 - sqrt(4*n**4 - 4*(n**2 - 2*n)*(n**2 - 2)))/(2*(n**2 - 2))
        nStar = star(n)
        a = aUpper - prec
        b = aUpper
        dZero = nStar - 1.4
    return listOfParams
\end{lstlisting}

\section{Sage Method For Computing Solutions to Trinomial Thue Equations}\label{sageMethod}

Here is the specific command which takes as input a degree $n$ and height $H$.  It finds every irreducible trinomial $F(x,y)$ with degree $n$ and height $H$, solves the Thue equation $|F(x,y)| = 1$, then stores the trinomials and their solution lists in a .csv file called \verb degree_n_height_H_thue_equations.csv .

\begin{lstlisting}[language = Python]
import itertools
import csv

R.<x> = ZZ[]

def TrinomialThueWriter(degree,height):
    filename = "thue_equation_solution_data/degree_{}_height_{}_thue_equations.csv".format(degree,height)
    columnHeads = ["Number of Solutions to |F(x,y)| = 1", "Leading Coefficient", "Middle Coefficient", "Constant Coefficient", "Middle Degree", "List of Solutions to |F(x,y)| = 1"]
    rows = []
    # Note that we only need to check positive leading coefficients since F(x,y) will have the same solutions as -F(x,y).
    # Note also that if the leading coefficient is larger than the absolute value of the constant coefficient, then we will have already computed the reciprocal polynomial F(y,x).  Hence, we can skip polynomials where the constant coefficient has absolute value less than the leading coefficient.
    for leadCoeff in range(1,height + 1):
        for midCoeff in itertools.chain(range(-height,0),range(1,height+1)):
            for constantCoeff in itertools.chain(range(-height,-leadCoeff+1), range(leadCoeff,height+1)):
                if (abs(leadCoeff) == height or abs(midCoeff)== height or abs(constantCoeff)== height) and (gcd(gcd(leadCoeff,midCoeff),constantCoeff)==1):
                    for midDegree in range(1,degree):
                        P = leadCoeff * x^degree + midCoeff * x^midDegree + constantCoeff
                        if P.is_irreducible():
                            thueInfo = gp.thueinit(P)
                            negSolns = gp.thue(thueInfo,-1)
                            posSolns = gp.thue(thueInfo,1)
                            totalSolns = len(negSolns)+len(posSolns)
                            rows.append([totalSolns, leadCoeff, midCoeff, constantCoeff, midDegree, gp.concat(posSolns,negSolns)])
    with open(filename,'w') as csvfile:
        csvwriter = csv.writer(csvfile)
        csvwriter.writerow(columnHeads)
        csvwriter.writerows(rows)
\end{lstlisting}

\end{document}